\documentclass[10pt]{amsart}
\usepackage{amsthm,hyperref}
\usepackage{times,mathptm}
\usepackage{amscd,amsfonts,amssymb}
\usepackage{graphicx,xy,tikz}
\usepackage{color}
\usepackage{amsthm}
\numberwithin{equation}{section}
\newtheorem{theorem}{Theorem}[section]
\newtheorem{lemma}[theorem]{Lemma}
\newtheorem{proposition}[theorem]{Proposition}

\theoremstyle{definition}

\newtheorem{definition}[theorem]{Defintion}

\begin{document}

\title[Existence and Uniqueness of Solutions for a Nonlinear Equation with Convection Term]%
{Existence and Uniqueness of Solutions for a Nonlinear Equation with Convection Term}
\author[Mustapha Ait Hammou]%
{Mustapha Ait Hammou*}

\newcommand{\acr}{\newline\indent}

\address{\llap{*\,}Laboratory LAMA\acr
                   Department of Mathematics\acr
                   Sidi Mohamed Ben Abdellah university\acr
                   Fez\acr
                   MOROCCO}
\email{mustapha.aithammou@usmba.ac.ma}



\subjclass[2010]{35J60, 47H11, 35D30, 35A01, 35A02} 
\keywords{Nonlinear elliptic equations, Topological degree, Weak solution, Existence and Uniqueness}

\begin{abstract}
In this paper, we consider the existence and uniqueness of weak solutions of a nonlinear elliptic equation with a variable exponent, a monotonic type operator and a convection term. With the topological degree theory, we prove the existence of at least one weak solution under some Leray-Lions and growth conditions. Moreover, we obtain the uniqueness of the solution of the problem under some additional assumptions. Our results generalize and improve existing results with another approach.
\end{abstract}

\maketitle

\section{Introduction}
In the present paper, we focus on the existence and uniqueness of a weak solution for a class of nonlinear elliptic boundary value problems in the framework of the Lebesgue and Sobolev spaces with variable exponent. The use of such spaces is justified by the study of several materials that present inhomogeneities and for which the context given by classical spaces is not adequate. Indeed, for such materials, the exponents involved in the constitutive law could be variable. The great attention on this topic is due to the many and various applications concerning thermorheological
fluids \cite{AR}, image restoration \cite{CLR}, electrorheological fluids \cite{R,SSE} and elastic materials \cite{Zh}.\par
Consider the following problem with a Neuman boundary condition
\begin{equation}\label{Pr1}
\left\{\begin{array}{ccc}
-div\;a(x,\nabla u)+\lambda|u|^{q(x)-2}u=f(x,u), &x\in\Omega&,\\
a(x,\nabla u).\eta=0, &x\in\partial\Omega&,
\end{array}\right.
\end{equation}
where $\lambda \in \mathbb{R}$, $\Omega\subset \mathbb{R}^N$ is a bounded domain, $a$ generates a nonlinear operator of Leray-Lions $A$ from
 $W_0^{1,p(\cdot)}(\Omega)$ to its dual $W^{-1,p'(\cdot)}(\Omega)$ which is defined by
$$A(u)=-diva(x,\nabla u),$$
$p(\cdot)$ and $ q(\cdot)$ are two variable exponents satisfying some conditions to be seen in the paper suite, $f$ is a Carath\'eodory function
satisfying a growth condition with a variable exponent that is suitably controlled by $p(\cdot)$ and $\eta$ is the outward unit normal to $\partial\Omega$.\par
In \cite{AO}, the authors showed the existence and uniqueness of a weak solution to the problem
$$ -\Delta_pu+m(x)|u|^{p-2}u=f(x,u) \mbox{ in } \mathbb{R}^N $$
which involves the $p$-Laplacian through Browder's theorem, where\\
$1<p<N, \;N\geq3$ and under some conditions for the functions $m$ and $f$.\\
When function $f$ is null and $a(x,\nabla u)=|\nabla u|^{q(x)-2}\nabla u$, the problem \eqref{Pr1} has been treated as an eigenvalue and eigenvector problem \cite{XQD}.\\
 Zhao et al. \cite{ZZX} have established the existence and the multiplicity of weak solutions of the boundary-value problem
 \begin{equation*}
\left\{\begin{array}{ccc}
-div\;a(x,\nabla u)+|u|^{p-2}u=\lambda f(x,u), &x\in\Omega&,\\
u(x)=constant,&x\in\partial\Omega&,\\
\int_{\partial\Omega}a(x,\nabla u).n\;ds=0,& &
\end{array}\right.
\end{equation*}
They found one nontrivial solution by the mountain pass lemma in \cite{PRS}, when the nonlinearity has a $(p-1)$-superlinear growth at infinity, and two nontrivial solutions by minimization and mountain pass when the nonlinear term has a $(p-1)$-sublinear growth at infinity.\\
For $\lambda=0$ and $a(x,\nabla u)=|\nabla u|^{q(x)-2}\nabla u$, P.S. Ilia\c{s} in \cite{I} and Fan and Zhang in \cite{FZ} gave several sufficient conditions for the existence of weak solutions for that problem under Dirichlet's boundary condition.\par
In this paper, we study problem \eqref{Pr1}, on the one hand as a kind of generalisation of the few previous works, and on the other hand, using another method based on the Topological Degree Theory developed by Berkovits \cite{Ber} for some classes of operators in Banach reflexive spaces (For some applications of this degree, the reader can see \cite{AA,AAB,AAK,NV}).\par
This document is organised as follows: Section 2 is reserved for some mathematical preliminaries. In section 3, we give our basic assumptions, some technical
lemmas and we give and prove our results of existence and uniqueness.
\section{Mathematical Preliminaries}
\subsection{Definitions and proposition}
Let us start with a short reminder of the classes of operators mentioned in the introduction and of an important proposition which will be the key to proving
the existence of at least one weak solution of problem \eqref{Pr1}.\\
Let $X$ be a real separable reflexive Banach space with dual $X^*$ and with continuous pairing $\langle.\;,\;.\rangle$ and let $\Omega$ be a nonempty subset of $X$. The symbol $\rightarrow (\rightharpoonup)$ stands for strong (weak) convergence.\\
Let  $Y$ be a real Banach space. We recall that a mapping $F:\Omega\subset X \rightarrow Y$ is \it bounded, \rm if it takes any bounded set into a bounded set. $F$ is said to be \it demicontinuous, \rm if for any $(u_n)\subset\Omega$, $u_n \rightarrow u$ implies $F(u_n)\rightharpoonup F(u)$. $F$ is said to be \it compact \rm if it is continuous and the image of any bounded set is relatively compact.A mapping $F:\Omega\subset X\rightarrow X^*$ is said to be \it of class $(S_+)$\rm, if for any $(u_n)\subset\Omega$ with $u_n \rightharpoonup u$ and $limsup\langle Fu_n,u_n - u\rangle\leq 0 $, it follows that $u_n \rightarrow u$. $F$ is said to be \it quasimonotone \rm, if for any $(u_n)\subset\Omega$ with $u_n\rightharpoonup u$, it follows that $limsup\langle Fu_n,u_n - u\rangle\geq 0$.\\
For any operator $F:\Omega\subset X\rightarrow X$ and any bounded operator $T:\Omega_1\subset X\rightarrow X^*$ such that $\Omega\subset\Omega_1$, we say that $F$ satisfies condition $(S_+)_T$, if for any $(u_n)\subset\Omega$ with $u_n \rightharpoonup u$, $y_n:=Tu_n\rightharpoonup y$ and
 $limsup\langle Fu_n,y_n-y\rangle\leq 0$, we have $u_n\rightarrow u$.
  For any $\Omega\subset X$, we consider the following classes of operators:
\begin{eqnarray*}
 \nonumber 
  \mathcal{F}_1(\Omega) &:=& \{F:\Omega\rightarrow X^*\mid F \mbox{ is bounded, demicontinuous and satifies condition }(S_+)\}, \\
  \mathcal{F}_{T,B}(\Omega) &:=& \{F:\Omega\rightarrow X\mid F \mbox{ is bounded, demicontinuous and satifies condition }(S_+)_T\}, \\
  \mathcal{F}_T(\Omega) &:=& \{F:\Omega\rightarrow X\mid F \mbox{ is demicontinuous and satifies condition }(S_+)_T\},\\
\end{eqnarray*}
 \begin{proposition}\label{l.proof}
 Let $S:X\rightarrow X^*$ and $T:X^*\rightarrow X$ be two operators bounded and continuous such that $S$ is quasimonotone and $T$ is of class $(S_+)$.
 If $$\Lambda:=\{v\in X^*|v+tSoTv=0 \mbox{ for some } t\in[0,1]\}$$ is bounded in $X^*$, then the equation $$v+SoTv=0$$ admits at lest one solution in $X^*.$
 \end{proposition}
 \begin{proof}
 Since $\Lambda$ is bounded in $X^*$, there exists $R>0$ such
that $$\|v\|_{X^*}<R \mbox{ for all } v\in \Lambda.$$
This says that
$$v+tSoTv\neq0 \mbox{ for all } v\in\partial B_R(0) \mbox{ and all } t\in[0,1]$$
where $B_R(0)$ is the ball of center $0$ and radius $R$ in $X^*$.\\
From \cite[Lemma 2.2 and 2.4]{Ber} it follows that
$$I+SoT\in \mathcal{F}_T(\overline{B_R(0)})\mbox{ and } I=LoT\in \mathcal{F}_T(\overline{B_R(0)}).$$
Since the operators $I$, $S$ and $T$ are bounded, $I+SoT$ is also bounded. We conclude that
$$I+SoT\in \mathcal{F}_{T,B}(\overline{B_R(0)})\mbox{ and } I\in\mathcal{F}_{T,B}(\overline{B_R(0)}).$$
Consider a homotopy $H:[0,1]\times\overline{B_R(0)}\rightarrow X^*$ given by
$$H(t,v):=v+tSoTv \mbox{ for } (t,v)\in[0,1]\times\overline{B_R(0)}.$$
Let us apply the homotopy invariance and normalization property  of the Berkovits degree (which we note $d$) introduced in \cite{Ber}, we get
$$d(I+SoT,B_R(0),0)=d(I,B_R(0),0)=1,$$
and hence there exists a point $v\in B_R(0)$ such that
$$v+SoTv=0.$$
 \end{proof}
 \subsection{Functional framework}
In the sequel, $\Omega$ is a smooth bounded domain in $\mathbb{R}^N\;(N\geq2)$.\par
In order to discuss problem \eqref{Pr1}, we start with the definition of the variable exponent Lebesgue spaces $L^{p(.)}(\Omega)$ and the variable
exponent Sobolev spaces $W_0^{1,p(.)}(\Omega)$, and some properties of them; for more details, see \cite{FZ1,KR}.\par
Let us denote $$C_+(\overline{\Omega})=\{h\in C(\overline{\Omega}):h(x)>1 \mbox{ for every } x\in\overline{\Omega}\}.$$
For any $h\in C_+(\overline{\Omega})$, we write
$$h^-:=\min_{x\in\overline{\Omega}}h(x)\;\;,\;\;h^+:=\max_{x\in\overline{\Omega}}h(x).$$
For any $p\in C_+(\overline{\Omega})$, we define the variable exponent Lebesgue space by
$$L^{p(\cdot)}(\Omega)=\{u;\ u:\Omega \rightarrow \mathbb{R} \mbox{  is measurable and }\rho_{p(\cdot)}(u)<\infty \}.$$
where
$$\rho_{p(\cdot)}(u)=\int_\Omega|u(x)|^{p(x)}\;dx.$$
We consider this space to be endowed with the so-called {\it Luxemburg norm}:
$$\|u\|_{p(\cdot)}=\inf\{{\lambda>0}: \rho_{p(.)}\Big(\frac{u}{\lambda}\Big)\leq 1\}.$$
We define the variable exponent Sobolev spaces $W^{1,p(\cdot)}(\Omega)$ by
$$W^{1,p(\cdot)}(\Omega)=\{u\in L^{p(\cdot)}(\Omega): |\nabla u | \in L^{p(\cdot)}(\Omega)\}$$
equipped with the norm
$$\|u\|_{W^{1,p(\cdot)}}=\|u\|_{p(\cdot)}+\|\nabla u\|_{p(\cdot)}.$$
The space $W_0^{1,p(\cdot)}(\Omega)$ is defined by the closure of $C_0^\infty(\Omega)$ in $W^{1,p(\cdot)}(\Omega)$. With these norms, the spaces
$L^{p(\cdot)}(\Omega),W^{1,p(\cdot)}(\Omega)$ and $W_0^{1,p(\cdot)}(\Omega)$ are separable reflexive Banach spaces.\par
The conjugate space of $L^{p(\cdot)}(\Omega)$ is $L^{p'(\cdot)}(\Omega)$ where $\frac{1}{p(x)}+\frac{1}{p'(x)}=1$.\\
For any $u\in L^{p(\cdot)}(\Omega)$ and $v\in L^{p'(\cdot)}(\Omega)$, H\"older inequality holds \cite[Theorem 2.1]{KR}:
\begin{equation}\label{Hol.in}
\left|\int_\Omega uv\;dx\right|\leq\left(\frac{1}{p^-}+\frac{1}{{p^{'}}^-}\right)\|u\|_{p(\cdot)}\|v\|_{p'(\cdot)}\leq2\|u\|_{p(\cdot)}\|v\|_{p'(\cdot)}.
\end{equation}
If $p(\cdot),\;q(\cdot)\in C_+(\overline{\Omega}),\; q(\cdot) \leq p(\cdot)$ a.e. in $\Omega$ then there exists the continuous embedding\\
$L^{p(\cdot)}(\Omega)\rightarrow L^{q(\cdot)}(\Omega)$.
\par
In this paper, we suppose that such that $p(\cdot)$ satisfies the log-H\"{o}lder continuity condition, i.e. there exists $C >0$ such that for all
$x,y \in \Omega$, $x\neq y$, one has
\begin{equation}\label{Hol}
  |p(x)-p(y)|\log\Big(e+\frac{1}{|x-y|}\Big)\leq C
\end{equation}
\par
An interesting feature of generalized variable exponent Sobolev space is that smooth functions are not dense in it without additional assumptions on the exponent $p(\cdot).$ However, when the exponent satisfies the $\log$-H\"older condition \eqref{Hol}, we recall the Poincar\'e inequality (see \cite[Theorem 8.2.4]{DHHR} and \cite[Theorem 2.7]{FZ1}) : there exists a constant $C>0$ depending only on $\Omega$ and the function $p$ such that
\begin{equation}\label{ptcar}
  \|u\|_{p(\cdot)}\leq C\|\nabla u\|_{p(\cdot)},\;\;\forall u\in W_0^{1,p(\cdot)}(\Omega),
\end{equation}
In particular, the space $W^{1,p(\cdot)}_{0}(\Omega)$ has a norm given by
$$\|u\|_{1,p(\cdot)}=\|\nabla u\|_{p(\cdot)},$$
which is equivalent to the norm $\|\cdot\|_{W^{1,p(\cdot)}}$.\\ Moreover, the embedding $W^{1,p(\cdot)}_{0}(\Omega)\hookrightarrow
L^{p(\cdot)}(\Omega)$ is compact (see \cite{KR}).\\
The space $(W_0^{1,p(\cdot)}(\Omega),\|\cdot\|_{1,p(\cdot)})$ is also a Banach space separable and reflexive .\\
\par
The dual space of $W_0^{1,p(\cdot)}(\Omega)$, denoted $W^{-1,p'(\cdot)}(\Omega)$, is equipped with the norm
$$\|v\|_{-1,p'(\cdot)}=\inf\{\|v_0\|_{p'(\cdot)}+\sum_{i=1}^{N}\|v_i\|_{p'(\cdot)}\},$$ where the infinimum is taken on all possible decompositions $v=v_0-div F$ with $v_0\in L^{p'(\cdot)}(\Omega)$ and $F=(v_1,...,v_N)\in (L^{p'(\cdot)}(\Omega))^N.$\\
\begin{proposition}\label{prop}\cite{FZ1}
Let $(u_n)\subset L^{p(\cdot)}(\Omega)$ and $u\in L^{p(\cdot)}(\Omega)$ , we have
  \begin{enumerate}
    \item $\|u\|_{p(\cdot)}>1\;\;\;\Rightarrow\;\;\;\|u\|_{p(\cdot)}^{p^-}\leq\rho_{p(\cdot)}(u)
\leq\|u\|_{p(\cdot)}^{p^+},$
    \item $\|u\|_{p(\cdot)}<1\;\;\;\Rightarrow\;\;\;\|u\|_{p(\cdot)}^{p^+}\leq
\rho_{p(\cdot)}(u)\leq\|u\|_{p(\cdot)}^{p^-},$
    \item $\lim_{n\rightarrow\infty}\|u_n-u\|_{p(\cdot)}=0 \;\;\;\Leftrightarrow\;\;\;\lim_{n\rightarrow\infty}\rho_{p(\cdot)}(u_n-u)=0,$
    \item $\|u\|_{p(\cdot)}\leq \rho_{p(\cdot)}(u)+1,$
    \item $\rho_{p(\cdot)}(u)\leq\|u\|_{p(\cdot)}^{p^-}+\|u\|_{p(\cdot)}^{p^+}.$
  \end{enumerate}
\end{proposition}

\section{Basic assumptions and Main Results}
Let $q\in C_+(\bar{\Omega})$, $1<q^-\leq q(x)\leq q^+<p^-\leq p(x)\leq p^+<\infty$ and\\
$f:\Omega\times\mathbb{R}\rightarrow \mathbb{R}$ be a real-valued function such that:
\begin{description}
  \item[($f_1$)] $f$ satisfies the Carath\'eodory condition, i.e. $f(.,\eta)$ is measurable on $\Omega$ for all $\eta\in\mathbb{R}$ and
  $f(x,.)$ is continuous on $\mathbb{R}$ for a.e. $x\in\Omega.$
  \item[($f_2$)] $f$ has the growth condition $$|f(x,\eta)|\leq c_1(k(x)+|\eta|^{r(x)-1})$$ for a.e. $x\in\Omega$ and all
  $\eta\in\mathbb{R}$, where $c_1$ is a positive constant, $k\in L^{p'(x)}(\Omega)$, $k(x)\geq 0$ and $r\in C_+(\bar{\Omega})$ with
   $2<r^-\leq r(x)\leq r^+ < p^-.$
\end{description}
Let $A:W_0^{1,p(\cdot)}(\Omega)\rightarrow W^{-1,p'(\cdot)}(\Omega)$ be the nonlinear operator of Leray-Lions, which is defined by
$$A(u)=-diva(x,\nabla u).$$i.e.
\begin{equation}\label{oper1}
  \langle A(u),v \rangle = \int_\Omega a(x,\nabla u).\nabla v \,dx \mbox { for all } v\in W_0^{1,p(.)}(\Omega).
\end{equation}
The function $a$ is assumed to satisfy the conditions:
\begin{description}
    \item[($A_1$)] $a:\Omega\times \mathbb{R}^N\rightarrow \mathbb{R}^N$ is a Caratheodory function,i.e., $a(.,\xi)$ is measurable on $\Omega$ for all $\xi\in\mathbb{R}^N$ and $a(x,.)$ is continuous on $\mathbb{R}^N$ for a.e. $x\in\Omega$.
   \item[($A_2$)] There exist a positive function $b(x)$ in $L^{p'(\cdot)}(\Omega)$ and constant $c>0$ such that $|a(x,\xi)|\leq b(x)+c |\xi|^{p(x)-1}$ for all $\xi\in\mathbb{R}^N$ and a.e. $x\in\Omega$.
   \item[($A_3$)] $(a(x,\xi)-a(x,\xi')).(\xi-\xi')>0$ a.e. $x\in\Omega$, for all $\xi,\xi'\in \mathbb{R}^N,\xi\neq\xi'$.
   \item[($A_4$)] There exist a constant $c'>0$ such that $a(x,\xi).\xi\geq c'|\xi|^{p(x)}$.
\end{description}
\begin{lemma}\label{lem1}(\cite[Lemma 4.3]{AAB}
  Under assumptions $(f_1)$ and $(f_2)$, the operator \\$S:W_0^{1,p(x)}(\Omega)\rightarrow W^{-1,p'(x)}(\Omega)$ setting by
  $$\langle Su,v\rangle=\int_\Omega(\lambda |u|^{q(x)-2}u-f(x,u))v dx,\;\;\ \forall u,v\in W_0^{1,p(x)}(\Omega)$$ is compact.
\end{lemma}
\begin{lemma}\label{lem3}
  Assume that the conditions $(A_1)-(A_4)$ hold. Then the operator $A$ defined by \eqref{oper1} is bounded, continuous, of type $(S_+)$ and coercive i.e. $\displaystyle\lim_{\|v\|_{1,p(\cdot)}\rightarrow\infty}\frac{<Av,v>}{\|v\|_{1,p(\cdot)}}=\infty.$
\end{lemma}
\begin{proof}
See the proof of Lemma 3.2 in \cite{AA} taking the weight function $w\equiv 1.$
\end{proof}

\par Let us first define the weak solution of problem \eqref{Pr1}.
\begin{definition}\label{Def}
  We call that $u\in W_0^{1,p(\cdot)}(\Omega)$ is a weak solution of \eqref{Pr1} if
  $$\int_{\Omega}a(x,\nabla u)\nabla v\;dx =\int_\Omega -\lambda|u|^{q(x)-2}uv\;dx+\int_{\Omega}f(x,u)v\;dx,\;\;\forall v\in W_0^{1,p(\cdot)}(\Omega).$$
\end{definition}
\begin{theorem}
We suppose that the assumptions $(f_1)$, $(f_2)$ and $(A_1)-(A_4)$ hold true, then there exists at least one weak solution of the problem (\ref{Pr1}) in $W_0^{1,p(\cdot)}(\Omega).$
\end{theorem}
 \begin{proof}
Let $A\mbox{ and }S:W_0^{1,p(\cdot)}(\Omega)\rightarrow W^{-1,p'(\cdot)}(\Omega)$ be as in \eqref{oper1} and Lemma \ref{lem1} respectively. Then $u\in W_0^{1,p(\cdot)}(\Omega)$ is a weak solution of \eqref{Pr1} if and only if
  \begin{equation}\label{eq2}
    Au=-Su
  \end{equation}
Thanks to the assumbtion $(A_3)$ and the properties of operator $A$ seen in Lemma \ref{lem3} and in view of Minty-Browder Theorem \cite[Theorem 26A]{Z}, the inverse operator $T:=A^{-1}:W^{-1,p'(\cdot)}(\Omega)\rightarrow W_0^{1,p(\cdot)}(\Omega)$ is bounded, continuous and of type $(S_+)$. Moreover, note from Lemma \ref{lem1} that the operator $S$ is bounded, continuous and quasimonotone.\\ Therefore, equation (\ref{eq2}) is equivalent to
\begin{equation}\label{eq3}
  u=Tv \mbox{ and } v+SoTv=0.
\end{equation}
To solve equation (\ref{eq3}), we will apply the Proposition \ref{l.proof}. It is sufficient to show that the set
$$\Lambda:=\{v\in W^{-1,p'(\cdot)}(\Omega)|v+tSoTv=0 \mbox{ for some } t\in[0,1]\}$$
is bounded.\par
Indeed, let $v\in \Lambda$ and set $u:=Tv$, then $\|Tv\|_{1,p(\cdot)}=\|\nabla u\|_{p(\cdot)}.$\par
If $\|\nabla u\|_{p(\cdot)}\leq 1$, then $\|Tv\|_{1,p(\cdot)}$ is bounded.\par
If $\|\nabla u\|_{p(\cdot)}>1$, then we have by Proposition \ref{prop}
\begin{equation}\label{1}
  \|Tv\|_{1,p(\cdot)}^{p^-}=\|\nabla u\|_{p(\cdot)}^{p-}\leq \rho_{p(\cdot)}(\nabla u).
\end{equation}
By the assumption $(A_4)$, we have
$$a(x,\nabla u).\nabla u\geq c'|\nabla u|^{p(x)}.$$
Then
\begin{eqnarray*}
\rho_{p(\cdot)}(\nabla u)&=&\int_\Omega |\nabla u|^{p(x)}dx \\ \nonumber
  &\leq&\frac{1}{c'}\int_\Omega a(x,\nabla u).\nabla u  \\ \nonumber
  &=&\frac{1}{c'} \langle Au,u\rangle  \\ \nonumber
  &=& \frac{1}{c'}\langle v,Tv\rangle \\ \nonumber
   &=& \frac{-t}{c'}\langle SoTv,Tv \rangle.
\end{eqnarray*}
This implies that
\begin{equation}\label{2}
  \rho_{p(\cdot)}(\nabla u)\leq \frac{t}{c'}\int_\Omega -\lambda |u|^{q(x)-2}u+f(x,u)u dx.
\end{equation}
We get, by the inequalities \eqref{1}, \eqref{2}, the growth condition the growth condition $(f_2)$, the H\"older inequality \eqref{Hol.in}, the inequality (5) of Proposition \ref{prop} and the Young inequality the estimate
\begin{eqnarray*}
  \|Tv\|_{1,p(\cdot)}^{p^-} &\leq& const(\lambda\rho_{q(\cdot)}(u)+\int_\Omega|k(x)u(x)|\;dx+\rho_{r(\cdot)}(u)) \\
  &\leq& const(\|u\|_{q(\cdot)}^{q^-}+\|u\|_{q(\cdot)}^{q^+}+\|k\|_{p'(\cdot)}\|u\|_{p(\cdot)}+\|u\|_{r(\cdot)}^{r^-}+\|u\|_{r(\cdot)}^{r^+})\\
  &\leq& const(\|u\|_{q(\cdot)}^{q^-}+\|u\|_{q(\cdot)}^{q^+}+\|u\|_{p(\cdot)}+\|u\|_{r(\cdot)}^{r^-}+\|u\|_{r(\cdot)}^{r^+}).
\end{eqnarray*}
From the Poincar\'e inequality (\ref{ptcar}) and the continuous embedding $L^{p(x)}\hookrightarrow L^{q(x)}$ and\\
$L^{p(x)}\hookrightarrow L^{r(x)}$, we can deduct the estimate
$$\|Tv\|_{1,p(\cdot)}^{p^-} \leq const(\|Tv\|_{1,p(\cdot)}^{q^+}+\|Tv\|_{1,p(\cdot)}+\|Tv\|_{1,p(\cdot)}^{r^+}).$$
It follows that $\{Tv|v\in \Lambda \}$ is bounded.\\
Since the operator $S$ is bounded, it is obvious from (\ref{eq3}) that the set $\Lambda$ is bounded in $W^{-1,p'(\cdot)}(\Omega).$\\
Hence, in virtu of Proposition \ref{l.proof}, the equation $v+SoTv$ have at lest one non trivial solution $\bar{v}$ in  $W^{-1,p'(\cdot)}(\Omega).$
So, $$\bar{u}=T\bar{v}$$ is a weak solution of \eqref{Pr1}.
\end{proof}
\par Next, we consider the uniqueness of solutions of \eqref{Pr1}. To this end, we also need the following hypothesis on the convection term:
\begin{description}
  \item[($f_3$)] There exists $c_2\geq0$ such that
  $$(f(x,t)-f(x,s))(t-s)\leq c_2|t-s|^{q(x)}$$ for a.e. $x\in\Omega$ and all
  $t,s\in\mathbb{R}$.
\end{description}
Our uniqueness result reads as follows.
\begin{theorem}
Assume that $(f_1)-(f_3)$ and $(A_1)-(A_4)$ hold. The weak solution of \eqref{Pr1} is unique provided
$$\frac{2^{q^+}c_2}{\lambda}<1.$$
\end{theorem}
\begin{proof}
  Let $u_1, u_2\in W_0^{1,p(\cdot)}(\Omega)$ be two weak solutions of \eqref{Pr1}. by choosing\\
   $v=u_1-u_2$ in the Definition \ref{Def}, we have
$$\int_{\Omega}a(x,\nabla u_1)\nabla(u_1-u_2)\;dx +\int_\Omega \lambda|u_1|^{q(x)-2}u_1(u_1-u_2)\;dx=\int_{\Omega}f(x,u_1)(u_1-u_2)\;dx$$
and
$$\int_{\Omega}a(x,\nabla u_2)\nabla(u_1-u_2)\;dx +\int_\Omega \lambda|u_2|^{q(x)-2}u_2(u_1-u_2)\;dx=\int_{\Omega}f(x,u_2)(u_1-u_2)\;dx$$
Subtracting the above two equations, we have
$$\int_{\Omega}(a(x,\nabla u_1)-a(x,\nabla u_2))\nabla(u_1-u_2)\;dx +\int_\Omega \lambda(|u_1|^{q(x)-2}u_1-|u_2|^{q(x)-2}u_2)(u_1-u_2)\;dx$$
$$=\int_{\Omega}(f(x,u_1)-f(x,u_2))(u_1-u_2)\;dx.$$
By assumption $(A_3)$, we have
$$(a(x,\nabla u_1)-a(x,\nabla u_2))\nabla(u_1-u_2)>0.$$
Moreover, since $q(x)\geq2$ , then we have the following inequality (see \cite{T}):
$$(|u_1|^{q(x)-2}u_1-|u_2|^{q(x)-2}u_2)(u_1-u_2)\geq(\frac{1}{2})^{q(x)}|u_1-u_2|^{q(x)}.$$
So, by using assumption $(f_3)$, we have
\begin{eqnarray*}
  \lambda(\frac{1}{2})^{q^+}\int_\Omega|u_1-u_2|^{q(x)}\;dx &\leq& \int_{\Omega}(f(x,u_1)-f(x,u_2))(u_1-u_2)\;dx \\
   &\leq& c_2\int_\Omega|u_1-u_2|^{q(x)}\;dx.
\end{eqnarray*}
Consequently, when $\frac{2^{q^+}c_2}{\lambda}<1$, it follows that $u_1=u_2$ and so the solution of \eqref{Pr1} is unique.
\end{proof}



\end{document}